\documentclass[12pt]{article}
\usepackage{amsfonts}
\newtheorem{theorem}{Theorem}[section]
\newtheorem{prop}[theorem]{Proposition}
\newtheorem{unnumber}{}

\newtheorem{prepf}[unnumber]{Proof}
\newenvironment{proof}{\prepf\rm}{\endprepf}

\begin{document}
\title{Filters, topologies, the Rado graph, and the Urysohn space}
\author{Peter J. Cameron\\\small{University of St Andrews}}
\date{To the memory of Anatoly Vershik}
\maketitle

\begin{abstract}
My work with Anatoly Vershik concerned automorphism groups of the Rado graph
and homeomorphism groups of the Urysohn space. This paper contains some further
thoughts on these issues, together with connections to topologies and filters
on countable sets.

\medskip



\end{abstract}

\section{Introduction}

I first met Anatoly Vershik at the European Congress of Mathematics in
Barcelona in 2000. I had just given a lecture on one of my favourite topics,
the Rado graph (also known as the Erd\H{o}s--R\'enyi random graph), when
Anatoly introduced himself to me and proceeded to give me a tutorial about
what became another of my favourite topics, the Urysohn metric space. We
wrote a paper~\cite{cv} about isometry groups of this space; the current
paper, among other things, throws a little more light on this topic.

In this paper I will outline a proof of a theorem about countable
topological spaces with primitive group of homeomorphisms. This leads to
a connection between the ``rational Urysohn space'' and the
``rational world'' (the latter being the term used by Peter Neumann to
describe $\mathbb{Q}$ as topological space). Filters arise in these topics,
and I show that the neighbourhood filter in the Rado graph is the universal
countable neighbourhood filter.
The paper also contains some remarks on topologies as relational structures.

I would like to thank Sam Tarzi for many discussions on some of the results
presented here.

\section{Imprimitive groups and topologies}

A permutation group $G$ on a set $\Omega$ is \emph{imprimitive} if it 
preserves a non-trivial equivalence relation on $\Omega$. (The trivial
equivalence relations are the relation of equality and the universal relation
with a single equivalence class. More generally, I will say that a relation
is trivial if it is invariant under the symmetric group.) The group $G$ is
\emph{primitive} otherwise.

A finite primitive permutation group has the following separation property:
if $G$ is primitive on $\Omega$, and $\Delta$ is a non-empty proper subset of
$\Omega$, then for distinct points $x,y\in\Omega$ there exists
$g\in G$ such that $xg\in\Delta$ and $yg\notin\Delta$. This is at
the base of many properties of primitive groups. But, in a talk in Oberwolfach
in the 1980s, Helmut Wielandt remarked that it fails for infinite primitive
groups. (For example, take $\Omega=\mathbb{Q}$, $G$ the group of
order-preserving permutations of $\mathbb{Q}$, and $\Delta$ the set of
positive rationals.) He proposed the following stronger notion.

A \emph{preorder} is a reflexive and transitive relation. Say that the
permutation group $G$ on $\Omega$ is \emph{strongly primitive} if it preserves
no non-trivial preorder on $\Omega$. Wielandt showed that some good properties
of finite primitive groups extend to strongly primitive groups but not to
primitive groups.

In fact there is a connection between preorders and topologies. Given a
topology $\mathcal{T}$ on $\Omega$, define a preorder $\to$ by the rule
that $x\to y$ if and only if every open set containing $x$ also contains $y$.
In the other direction, if $\to$ is a preorder, define a topology by the
rule that the sets $U_x=\{y:x\to y\}$ form a base for the topology.

Now the map preorder~$\to$~topology~$\to$~preorder is always the identity,
but the map topology~$\to$~preorder~$\to$~topology may not be, though it
always produces a stronger topology. We say that a topology is \emph{relational}
if it is fixed by this map.

Then, for example, every finite topology is relational.

Primitivity and Wielandt's strengthening of it are connected to separation
properties of $G$-invariant topologies.

\begin{prop}
Let $G$ be a transitive permutation group on $\Omega$. Then $G$ is primitive
if and only if every non-trivial $G$-invariant topology is T0, and $G$ is
strongly primitive if and only if every non-trivial $G$-invariant topology
is T1.
\end{prop}

Note that $G$ preserves a non-T0 (resp.~non-T1) topology if and only if it
preserves a relational topology with the same property.

Note also that the classes of finite equivalence relations and finite preorders
are \emph{Fra\"{\i}ss\'e classes}, so there exist countable universal 
homogeneous relational topologies of each of these types. The countable
homogeneous equivalence relation simply corresponds to a partition with
infinitely many infinite parts; the preorder is obtained from the countable
homogeneous partial order by inflating each point to an infinite set.

\section{Topologies and filters}

The celebrated O'Nan--Scott Theorem on finite primitive permutation groups
was originally regarded by its authors as a description of the maximal
subgroups of finite symmetric groups. Note that every proper subgroup of such
a group is contained in a maximal subgroup.

The situation is different for infinite symmetric groups: for example, a 
countable subgroup of the symmetric group of countable degree is always
contained in a larger countable subgroup, and so cannot be maximal. It is not
known whether any proper subgroup is contained in a maximal subgroup.

In 1990, Dugald Macpherson and Cheryl Praeger published a paper~\cite{mp}
proving that a subgroup of the symmetric group of countable degree which is
not highly transitive is contained in a maximal subgroup. (A permutation group
is \emph{highly transitive} if it acts transitively on the set of $n$-tuples
of distinct points of its domain for every natural number $n$.) Their proof
makes extensive use of tools from model theory including the notions of 
stability and Morley rank and theorems of Ehrenfeucht and Mostowski,
Ryll-Nardzewski, and Cherlin, Harrington and Lachlan. 

A \emph{filter} on $\Omega$ is a collection of subsets of $\Omega$ which is
closed upwards and closed under finite intersections. It is non-trivial
if it does not contain the empty set. (This is slightly at variance with the
earlier interpretation of non-triviality, in that the cofinite subsets of an
infinite set form a filter invariant under the symmetric group.) The argument
leans on the fact that the automorphism group of an \emph{ultrafilter} (a
maximal non-trivial filter) is a maximal subgroup of the symmetric group. Note
however that not every proper subgroup is contained in the stabiliser of an
ultrafilter: the referee points out that if a regular subgroup $G$ (acting on
itself by right multiplication) fixed an ultrafilter, there would be an
invariant multiplicative norm on $G$.

An important step in the Macpherson--Praeger proof was the following theorem.

\begin{theorem}
Let $G$ be a subgroup of the symmetric group on a countable set $\Omega$. Then
$G$ preserves a non-trivial topology if and only if $G$ preserves a
non-trivial filter.
\end{theorem}

I was able to simplify this theorem; I sketch the argument here. See the
paper~\cite{topsym} for further details.

Just two constructions of filters from topologies are used in the proof:
\begin{itemize}
\item the sets containing finite intersections of dense open sets;
\item the complements of finite unions of discrete sets.
\end{itemize}
The proof involves showing that, in a non-trivial topology invariant under the
primitive group $G$, one or other of these must be a non-trivial filter.

Let $(\Omega,\mathcal{T})$ be a  topology preserved by a primitive group.
If it contains a non-empty finite open set, then it is discrete (since a
minimal open set is a block of imprimitivity). So we can suppose that this is
not the case. Also, if $\Omega$ is a union of finitely many discrete subsets,
then again it is discrete.

Define a graph $\Gamma$ on $\Omega$ by joining $x$ to $y$ if there exist
disjoint open sets containing $x$ and $y$. If $\Gamma$ contains no infinite
clique, then every finite clique is contained in a finite maximal clique,
and it can be shown that $\Gamma$ is complete multipartite, with parts of size
$m$, for some $m$; since it has a primitive automorphism group, we must have
$m=1$, so that any two non-empty open sets intersect. Then the collection of
sets containing a non-empty open set is a non-trivial filter.

On the other hand, if $\Gamma$ contains an infinite clique, then the induced
topology on it is Hausdorff, so $\Gamma$ contains an infinite discrete subset
(an exercise in Sierpi\'nski's book~\cite{sierpinski}), so the collection of
complements of discrete sets is a non-trivial filter.

\section{Mekler's theorem}

A simple property of topologies and filters with primitive automorphism groups
is given by the following result. A \emph{moiety} in a countable set is an
infinite and co-infinite subset.

\begin{prop}
Let $G$ be a primitive permutation group on the countable set $\Omega$.
\begin{enumerate}
\item $G$ preserves a non-trivial topology if and only if there is a moiety
$\Delta$ of $\Omega$ such that, for any $g_1,\ldots,g_n\in G$, the set
$\Delta g_1\cap\cdots\cap\Delta g_n$ is empty or infinite.
\item $G$ preserves a non-trivial filter if and only if there is a moiety
$\Delta$ of $\Omega$ such that, for any $g_1,\ldots,g_n\in G$, the set
$\Delta g_1\cap\cdots\cap\Delta g_n$ is infinite.
\end{enumerate}
\label{p:filtop}
\end{prop}

\begin{proof}
(a) Every non-empty open set is infinite; take $\Delta$ to be a non-cofinite
open set (this exists since the topology is non-trivial). Conversely, if 
$\Delta$ is such a set, then the class of all non-empty intersections of
finitely many translates of $\Delta$ is a basis for a non-trivial topology.

\medskip

(b) By Neumann's separation lemma~\cite[Lemma 2.3]{neumann}, no filter
invariant under an infinite transitive group can contain a finite set; so take
$\Delta$ to be any non-cofinite set in the filter. Conversely, if $\Delta$ has
the property in the Proposition, the collection of sets which contain a finite
intersection of translates of $\Delta$ is a non-trivial filter.
\end{proof}

The first part of this proposition is reminiscent of a celebrated theorem
of Mekler~\cite{mekler} (a simplified proof was given by Truss~\cite{truss}).
Using the terminology of Neumann~\cite{rational}, we use the term
\emph{rational world} to mean a countable topological space homeomorphic
to~$\mathbb{Q}$.

\begin{theorem}
Let $G$ be a countable permutation group on a countable set $\Omega$. Then
$G$ is a subgroup of the homeomorphism group of the rational world if and 
only if it satisfies the conditions of Proposition~\ref{p:filtop}(a).
\end{theorem}

This shows that any countable group which acts primitively on a non-trivial
countable topological space has the same action on the rational world. 
Indeed, primitivity is not required as long as Mekler's condition is
satified. So the existence of a transitive action of the infinite cyclic group
on the rational world~\cite{rational} follows from its action on the countable
Urysohn space~\cite{cv}. (The referee points out that this can be seen more
easily: the theorem of Sierpi\'nski~\cite{s:rat} shows that the rational
Urysohn space is homeomorphic to the rational world, even though the metrics
are quite different.

\section{Neighbourhood filters}

We have seen a connection between filters and topologies on a countable set.
Filters are also connected with graphs. 

Let $\Gamma$ be a graph on a countable vertex set $\Omega$. We define the
\emph{neighbourhood filter} of $\Gamma$ to be the filter $\mathcal{F}_\Gamma$
generated by the neighbourhoods $\Gamma(v)=\{w\in\Omega: v\sim w\}$, where 
$\sim$ denotes adjacency in the graph. (The filter generated by a collection
$\mathcal{C}$ of sets consists of all sets containing a finite intersection
of members of $\mathcal{C}$.)

This definition uses open neighbourhoods $\Gamma(v)$. Under weak conditions
on $\Gamma$, the closed neighbourhoods $\bar\Gamma(v)=\{v\}\cup\Gamma(v)$
generate the same filter:

\begin{prop}
Suppose that $\Gamma$ contains no dominant vertex (that is, every vertex has
a non-neighbour). Then the closed neighbourhoods in $\Gamma$ generate the
filter $\mathcal{F}_\Gamma$.
\end{prop}

For $\Gamma(v)\subseteq \bar\Gamma(v)$, and if $w\not\sim v$ then
$\bar\Gamma(v)\cap\bar\Gamma(w)\subseteq\Gamma(v)$.

Of course, if $\Gamma$ contains two vertices with no common neighbour, then
$\mathcal{F}_\Gamma$ contains the empty set, and so is trivial. However, there
is a nice characterisation of graphs whose neighbourhood filter is nontrivial.
Let $R$ be the Rado graph or (Erd\H{o}s--R\'enyi) random graph~\cite{randgr}.

\begin{theorem}
The following three conditions on a countable graph $\Gamma$ are equivalent:
\begin{enumerate}
\item $\mathcal{F}_\Gamma$ is nontrivial;
\item $\Gamma$ contains $R$ as a spanning subgraph;
\item $\mathcal{F}_\Gamma\subseteq\mathcal{F}_R$.
\end{enumerate}
\end{theorem}

\begin{proof}
The filter $\mathcal{F}_\Gamma$ is non-trivial if and only if any finitely
many neighbourhoods in $\Gamma$ have non-empty intersection. Now we can
define an embedding of $R$ into $\Gamma$ by modifying the standard
back-and-forth proof by going forth only. Thus (a) and (b) are equivalent.

If $\Gamma$ contains $R$ as a spanning subgraph, then $R(v)\subseteq\Gamma(v)$,
so (c) holds. Conversely, the proof that (a) and (b) are equivalent and 
standard properties of $R$ show that $\mathcal{F}_R$ is non-trivial.
\end{proof}

Thus, in some sense, $\mathcal{F}_R$ is the ``maximal'' neighbourhood filter.
But some care is required, since we can have two filters isomorphic to
$\mathcal{F}_R$, one properly containing the other.

To see this, let $T$ be the result of a random $3$-colouring of the edges of
the countable complete graph with colours red, green and blue (say). Let $R_1$
be the red subgraph and $R_2$ the red-green subgraph. Then, for every vertex
$v$, we have $R_1(v)\subseteq R_2(v)$;
so $\mathcal{F}_{R_2}\subseteq\mathcal{F}_{R_1}$. To show that the inequality
is strict, suppose for a contradiction that $R_1(v)$ belings to
$\mathcal{F}_{R_2}$. Then there are vertices $w_1,\ldots,w_n$ such that
\[R_2(w_1)\cap\cdots\cap R_2(w_n)\subseteq R_1(v).\]
But this is impossible, since the green graph is isomorphic to $R$, and so
there is a vertex $x$ joined to all of $v,w_1,\ldots,w_r$ by green edges;
thus $x$ belongs to the left-hand intersection above but not to $R_1(v)$.

One can extend this argument to produce countable chains of filters isomorphic
to $\mathcal{F}_R$.

\paragraph{Acknowledgement} I am grateful to the referee for helpful comments.

\end{document}